\documentclass[11pt]{article}

\usepackage[margin=1.0in]{geometry}

\usepackage{url}
\usepackage{enumitem}

\usepackage{amsmath,amssymb,euscript}
\usepackage{dsfont}

\usepackage{graphicx}
\usepackage{color}

\usepackage{ntheorem}

\usepackage{tikz}

\newenvironment{proof}[1][{}]{
  \begin{trivlist}\item[]\textit{Proof #1}\quad}%
  {\hfill\hspace*{\fill}~$\square$\end{trivlist}}

\theorembodyfont{\itshape}
\newtheorem{thm}{Theorem}
\newtheorem{prop}[thm]{Proposition}

\newtheorem{lem}[thm]{Lemma}

\newtheorem{cor}[thm]{Corollary}

\newtheorem*{rmk}[thm]{Remark} 

\theoremstyle{empty}

\definecolor{turquoise}{cmyk}{0.65,0,0.1,0.1}
\definecolor{darkgreen}{cmyk}{0.87,0.5,0.81,0.66}

\newcommand{\rawdef}[1]{\emph{#1}} 
\newcommand{\defn}[1]{\rawdef{#1}\index{#1}}

\newcommand{\Corref}[1]{Corollary~\ref{#1}}

\newcommand{\Eqnref}[1]{Equation~\eqref{#1}}

\newcommand{\Lemref}[1]{Lemma~\ref{#1}}

\newcommand{\Secref}[1]{Section~\ref{#1}}

\newcommand{\Thmref}[1]{Theorem~\ref{#1}}
\newcommand{\Propref}[1]{Proposition~\ref{#1}}

\newcommand{\cinfty}{C^\infty}

\DeclareMathOperator{\convh}{conv}
\newcommand{\convhull}[1]{\convh(#1)}

\newcommand{\R}{\mathbb{R}}

\newcommand{\E}{\mathbb{E}}

\DeclareMathOperator{\Grad}{grad}
\newcommand{\grad}{\Grad}

\newcommand{\norm}[1]{\left|#1\right|}

\newcommand{\abs}[1]{\left|#1\right|}

\newcommand{\bracketprod}[2]{\left\langle#1,#2\right\rangle}

\newcommand{\bdry}[1]{\partial{#1}}

\newcommand{\gdist}{d} 

\newcommand{\gdistG}[1]{\gdist_{#1}} 
\newcommand{\distG}[3]{\gdist_{#1}(#2,#3)}

\newcommand{\close}[1]{\overline{#1}} 

\DeclareMathOperator{\interior}{int}
\newcommand{\intr}[1]{\interior(#1)}

\newcommand{\spaceball}[3]{B_{#1}(#2,#3)} 
\newcommand{\cspaceball}[3]{\close{B}_{#1}(#2,#3)} 

\DeclareMathOperator{\aff}{aff} 
\newcommand{\affhull}[1]{\aff(#1)}

\newcommand{\splxs}{\sigma}
\newcommand{\splxt}{\tau}

\newcommand{\tsplxs}{\tilde{\splxs}}
\newcommand{\tsplxt}{\tilde{\splxt}}

\newcommand{\tp}{\tilde{p}}
\newcommand{\tq}{\tilde{q}}

\newcommand{\gthickness}{t} 

\newcommand{\gsplxalt}{a}
\newcommand{\glongedge}{L}

\newcommand{\man}{M}

\newcommand{\gdistEn}{\gdistG{\E^n}} 
\newcommand{\distEn}[2]{\distG{\E^n}{#1}{#2}}

\newcommand{\ballEn}[2]{\spaceball{\E^n}{#1}{#2}} 

\newcommand{\genfct}[1]{\mathcal{E}_{#1}}
\newcommand{\genfctbc}{\genfct{\ggbarycoord}}

\newcommand{\enfctbc}[1]{\genfctbc(#1)}
\newcommand{\genfctmu}{\genfct{\mu}}
\newcommand{\enfctmu}[1]{\genfctmu(#1)}

\newcommand{\gdistM}{\gdistG{\man}} 
\newcommand{\distM}[2]{\distG{\man}{#1}{#2}}

\newcommand{\ballM}[2]{\spaceball{\man}{#1}{#2}}
\newcommand{\cballM}[2]{\cspaceball{\man}{#1}{#2}}

\newcommand{\injradM}{\iota_{\man}}

\newcommand{\curvupbnd}{{\Lambda_{u}}}

\newcommand{\curvlowbnd}{{\Lambda_{\ell}}}
\newcommand{\curvabsbnd}{\Lambda}

\newcommand{\stdsplx}[1]{\boldsymbol{\Delta}^{\!#1}}
\newcommand{\stdsplxn}{\stdsplx{n}}

\newcommand{\riemsplxs}{\gsplxs^{}_{\!\!M}}
\newcommand{\riemsplxt}{\gsplxt^{}_{\!M}}

\newcommand{\riemsplxeta}{\gsplxeta^{}_{M}}

\newcommand{\ggbarycoord}{\lambda} 
\newcommand{\gbarycoord}[1]{\ggbarycoord_{#1}} 

\newcommand{\splxsE}{\gsplxs^{}_{\E^n}}
\newcommand{\tsplxsE}{\gtsplxs^{}_{\E^n}}

\newcommand{\tsplxtE}{\gtsplxt^{}_{\E^n}}

\DeclareMathOperator{\arccosh}{arccosh}

\newcommand{\cnstcurv}{\kappa}
\newcommand{\seccurv}{K}

\newcommand{\gsplxs}{\boldsymbol{\sigma}}
\newcommand{\gsplxt}{\boldsymbol{\tau}}

\newcommand{\gsplxeta}{\boldsymbol{\eta}}

\newcommand{\gtsplxs}{\boldsymbol{\tilde{\sigma}}}
\newcommand{\gtsplxt}{\boldsymbol{\tilde{\tau}}}

\newcommand{\rl}{\sqrt{\curvabsbnd}}

\newcommand{\Hk}{\mathbb{H}_{\cnstcurv}}
\newcommand{\Hl}{\mathbb{H}_{\curvlowbnd}}
\newcommand{\Hu}{\mathbb{H}_{\curvupbnd}}

\newcommand{\mup}{\mu_+(\man)}
\newcommand{\mun}{\mu_-(\man)}

\DeclareMathOperator{\supp}{supp} 

\newcommand{\tl}{\vartheta_\ell}
\newcommand{\tu}{\vartheta_u}

\newcommand{\tr}{\tilde{r}}
\newcommand{\tR}{\tilde{R}}
\newcommand{\tb}{\tilde{b}}
\newcommand{\tv}{\tilde{v}}

\usepackage[pagebackref=true]{hyperref}

\title{Barycentric coordinate neighbourhoods in Riemannian manifolds
}

\author{Ramsay Dyer\thanks{{\tt yasmar@gmail.com}}
        \and
        Gert Vegter\thanks{Johann Bernoulli Institute,
        Rijksuniversiteit Groningen} \thanks{{\tt g.vegter@rug.nl}}
        \and
        Mathijs Wintraecken
        \thanks{INRIA Sophia Antipolis}
        \thanks{{\tt m.h.m.j.wintraecken@gmail.com}}}

\begin{document}

\pagenumbering{roman}
\maketitle

%

\begin{abstract}
  We quantify conditions that ensure that a signed measure on a
  Riemannian manifold has a well defined centre of mass.  We then use
  this result to quantify the extent of a neighbourhood on which the
  Riemannian barycentric coordinates of a set of $n+1$ points on an
  $n$-manifold provide a true coordinate chart, i.e., the barycentric
  coordinates provide a diffeomorphism between a neighbourhood of a
  Euclidean simplex, and a neighbourhood containing the points on the
  manifold. 
\end{abstract}


\thispagestyle{empty}

\tableofcontents

\clearpage
\pagenumbering{arabic}
%
%

The focus of this work is on identifying conditions under which
barycentric coordinates will parameterise a neighbourhood of a simplex
in a Riemannian manifold, and specifically on quantifying the extent
of the permissible neighbourhood. 
Recall that if $\splxs=\{p_0,\ldots,p_n\} \subset \E^n$ defines a
\emph{nondegenerate} simplex in Euclidean space\footnote{We consider
  $\E^n$ to have an affine structure, and Euclidean distance function,
  but no canonical origin or coordinate system.}, i.e.,
$\aff(\splxs)= \E^n$, where $\aff(\splxs)$ is the affine hull of
$\splxs$, then $\E^n$ can be parameterised by  the barycentric
coordinate functions; any point $u \in \E^n$ is uniquely defined by a
set of $(n+1)$ barycentric coordinates, $(\lambda_i)$ that satisfy
$\sum \lambda_i = 1$. The point $u$ corresponding to $(\lambda_i)$ can
be identified as the point that minimises the function
\begin{equation}
  \label{eq:enfct.eucl}
  x \mapsto \tfrac12\sum_{i=0}^n \lambda_i \distEn{x}{p_i}^2,
\end{equation}
where $\distEn{x}{y}$ is the Euclidean distance function. In this way
the barycentric coordinates define the familiar affine functions that
satisfy $\lambda_i(p_j)=\delta_{ij}$.

We can view a given set of barycentric coordinates
$\ggbarycoord = (\gbarycoord{0},\ldots,\gbarycoord{n})$ as a point in
$\R^{n+1}$. The set $\stdsplxn$ of all points in $\R^{n+1}$ with
non-negative coefficients that sum to $1$ is called the \defn{standard
  Euclidean $n$-simplex}. Thus the minimisation of the
function~\eqref{eq:enfct.eucl} defines a \defn{barycentric coordinate
  map} from the the plane
$\{\lambda \in \R^{n+1} \mid \sum \lambda_i = 1 \}$ to $\E^n$ that
brings the standard Euclidean simplex to the Euclidean simplex
$\splxsE = \convhull\splxs \subset \E^n$.

This same technique can be used to define barycentric coordinate
neighbourhoods in a Riemannian manifold $\man$. If
$\splxs = \{p_i\} \subset \man$, we consider the function
\eqref{eq:enfct.eucl} except that the Euclidean distance $\gdistEn$ is
replaced by $\gdistM$, the intrinsic metric of the manifold. The first
problem is to identify conditions under which this function admits a
unique minimum. This is a particular case of a more general problem of
identifying conditions under which the \defn{Riemannian centre of
  mass} of a measure is well defined. There are several demonstrations
of the existence and uniqueness of Riemannian centres of mass. 

The standard reference for the subject is the excellent exposition by
Karcher \cite{Karcher} (see also the recent note outlining the earlier
history of the subject \cite{karcher2014hist}). He uses Jacobi field
estimates to demonstrate that if the support of a measure $\mu$ is
contained in a sufficiently small convex ball, then the \defn{energy
  function} $x \mapsto \tfrac12\int \distM{x}{y}^2\,d\mu(y)$ is
convex, and since the gradient is pointing outward on the boundary of
the ball, there is a unique minimum in the ball's interior.

Kendall \cite{Kendall} gave a different proof based on the concept of
``convex geometry'' which arose in earlier work on Dirichlet
problems. Kendal was able to relax Karcher's bounds on the size of the
ball that must contain the support of the measure. 

Groisser \cite{groisser2004} gives a proof based on a contraction
mapping argument, using a map defined from the gradient of the above
mentioned energy function. This yields a constructive algorithm for
finding the centre of mass, but the constraint on the extent of the
admissible neighbourhood is more restrictive than Karcher's. 

These works all assume an unsigned measure, so they do not apply to
situations where negative weights are involved.  The existence and
uniqueness of the Riemannian centre of mass with non-negative weights
at least allows us to define a ``filled in'' geometric \defn{Riemannian
  simplex} associated to a finite set of points with small
diameter. If $\splxs \subset \man$ is a set of $n+1$ points in a
Riemannian $n$-manifold, and $\splxs$ is contained in a sufficiently small
convex ball $B$, then we can define the Riemannian simplex
$\riemsplxs$ as the image of the standard simplex $\stdsplxn$ under
the barycentric coordinate map.

In order to make the concept of Riemannian simplices useful for
triangulating manifolds, we need to ensure that the barycentric
coordinate map is not only well-defined, but also an embedding. When
this is the case, we say that the Riemannian simplex is
\defn{nondegenerate}.  We recently established conditions, based on
the local curvatures of the manifold and the quality of the point set
$\sigma$, that ensure that the barycentric coordinate map is an
embedding \cite{RiemSimp2015}. 

Although we used the result on nondegenerate Riemannian simplices to
establish criteria for ensuring that a simplicial complex triangulates
a Riemannian manifold, we did not establish some simple properties
that we expect from geometric simplices. Specifically, if $\riemsplxs$
and $\riemsplxt$ are two Riemannian $n$-simplices in an $n$-manifold,
and they have $n$ vertices in common, and agree with a local
orientation of the manifold, then is $\riemsplxt \cap \riemsplxs$
exactly the Riemannian simplex that is the common facet? Of course it
is impossible to ensure a triangulation by Riemannian simplices
without ensuring this property, but in \cite{RiemSimp2015} we did not
show this property for two isolated simplices.

This is one of the motivations for establishing the barycentric
coordinate neighbourhoods that are the focus of the current work. If
the barycentric neighbourhood defined by $\riemsplxs$ encompasses
$\riemsplxt$ and vice versa, then the above problem is easily
resolved, as discussed in \Secref{sec:shared.facet}.

In order to establish such barycentric coordinate neighbourhoods we
must first ensure that the barycentric coordinate map is well defined
when the weights may be negative. To this end we give an alternate
proof of the existence and uniqueness of Riemannian centres of
mass. The result, \Thmref{thm:UniquenessKMnegative}, can accommodate 
signed measures, i.e., mass distributions that include negative
weights. This is the main technical contribution of this paper.

A demonstration of the existence and uniqueness of Riemannian centres
of mass for signed measures has already been published by Sander
\cite[Theorem~3.19]{Sander2}. He was motivated by a desire to describe
higher order interpolation on Riemannian manifolds, which needs to
accommodate negative interpolation weights. Sander's proof is modelled
on Groisser's proof for unsigned measures \cite{groisser2004}. The
result is stated in terms of several parameters governing the
neighbourhood of validity, and they are bounded in an intricate
manner. We are specifically interested in obtaining an explicit bound on
the extent of the admissible neighbourhood in terms of primitive
properties of the measure and bounds on the sectional curvatures. Such
an explicit bound is not easily obtained from Sander's result.

We provide an elementary proof of the existence and uniqueness
of Riemannian centres of mass for signed measures that is in the
spirit of Karcher's demonstration for the unsigned case. However,
instead of using Jacobi field estimates to demonstrate the convexity
of the energy function, we rely on a series expansion of the cosine
rule in constant curvature spaces. At the expense of this brute force
calculation, we obtain a concise and elementary
demonstration. Furthermore, our bounds on the admissible neighbourhood
are explicit, and they can be reduced to Karcher's bounds in the case where
the measure is unsigned.  This is especially evident in the
formulation of the result expressed in
\Corref{cor:simple.signed.measure}.

In \Secref{sec:background} we briefly review the tools we will use for
our exposition, including a recap of the main ideas in defining
Riemannian centres of mass and barycentric coordinates. We demonstrate
our main theorem on Riemannian centres of mass of signed measures in
\Secref{sec:NegWeight}, and \Secref{sec:barycoords} is devoted to
exploiting this for the development of barycentric coordinate
neighbourhoods that motivates this work.

In order to obtain true coordinate neighbourhoods we need to ensure
not only that the barycentric coordinate map is well defined, but also
that it is an embedding. We observe in \Secref{sec:barycoords} that we
can use the arguments in \cite{RiemSimp2015} verbatim since they did
not rely on the sign of the weights, only on the convexity of the
energy function, but this is established as part of the demonstration
of \Thmref{thm:UniquenessKMnegative}. We obtain
\Thmref{thm:bary.coord.nbhd.extent} and \Propref{prop:contained.ball}
which describe the extent of admissible barycentric coordinate
neighbourhoods. Finally, in \Secref{sec:shared.facet} we make the
observation that suitably constrained compatibly oriented Riemannian
$n$-simplices that share $n$-vertices intersect only on their shared
facet.


%

\section{Background and notation}
\label{sec:background}

In this work $\man$ refers to a $\cinfty$ Riemannian manifold (without
boundary) of dimension $n$. The adjective \defn{smooth} is synonymous
with $\cinfty$. A function defined on a closed set $A \subset \man$ is
\defn{smooth} if it can be extended to a smooth function on an open
neighbourhood of $A$.  The distance between $x,y \in \man$ is denoted
$\distM{x}{y}$, and
$\ballM{c}{r} = \{x \in \man \mid \distM{c}{x} < r\}$ is the open
geodesic ball of radius $r$ centred at $c\in \man$. The topological
closure of a set $B \subset \man$ is denoted $\close{B}$. A geodesic
is \defn{minimising} if it is the shortest path between any two of its
points (all geodesics are locally minimising). A
(geodesic) \defn{segment} is the trace of a minimising geodesic.  We
sometimes abuse notation and terminology by identifying a geodesic
with its trace.

The injectivity radius $\injradM$ is the supremum of the distances $r$
such that any two points $x,y \in \man$ with $\distM{x}{y}<r$ have a
unique minimising geodesic connecting them. This number is positive if
$\man$ is compact. For any $x \in \man$, the \defn{exponential map}
$\exp_x\colon T_x\man \to \man$ maps a vector%
\footnote{In general $\exp_x$ may not be defined on all of $T_x\man$,
  unless $\man$ is \defn{complete}, but this detail is not a concern
  for us.}  $v$ in the tangent space $T_x\man$ to the point
$\gamma(\norm{v}) \in \man$, where $\gamma$ is the unique unit-speed
geodesic emenating from $x$ with $\gamma'(0) = v/\norm{v}$, and
$\norm{v} = \bracketprod{v}{v}^{1/2}$ is the norm defined by the
Riemannian metric. If $r < \injradM$, then $\exp_x$ is a
diffeomorphism (onto its image) when restricted to
$\spaceball{T_x\man}{x}{r}$. Identifying $T_x\man$ with $\R^n$, this
yields a local coordinate system called \defn{Riemann normal
  coordinates}.

For any $x \in \man$, and any 2-dimensional subspace
$H \subseteq T_x\man$ we associate a \defn{sectional curvature}, which
is the Gaussian curvature at $x$ of the surface $\exp_x(H)$. We denote
by $K$ the sectional curvature function, but we are concerned only
with bounds on $K$, never individual values.

\subsection{Convexity}

A set $B \subseteq \man$ is \defn{convex} if for all $x,y \in B$ there
is a unique minimising geodesic in $\man$ connecting $x$ to $y$, and
this geodesic is contained in $B$. If $\curvupbnd$ is an upper bound
on the sectional curvatures of $\man$, then for any point $c \in \man$
the ball $\cballM{c}{r}$ is convex if
$r <  \min \bigl \{ \frac{\injradM}{2},
    \frac{\pi}{2\sqrt{\curvupbnd}} \bigr \}$,
where we set $\frac{\pi}{2\sqrt{\curvupbnd}}=\infty$ if
$\curvupbnd \leq 0$ (see \cite[Thm. IX.6.1]{Chavel}). 

A function $f\colon A \to \R$, with $A \subseteq \man$ is
\defn{convex} if for any $p \in A$ and geodesic $\gamma$ with
$\gamma(0) = p$, we have $\frac{d^2}{dt^2}f(\gamma(t))|_{t=0} \geq 0$.
We say $f$ is \defn{strictly convex} if this inequality is strict. If
$A$ is compact and convex with nonempty interior, and $f$ is strictly
convex with $\grad f$ pointing outwards on $\bdry{A}$ (i.e., for any
$x \in \bdry{A}$, $y \in \intr{A}$, and $\gamma$ a minimising geodesic
from $x$ to $y$, we have $\bracketprod{\gamma'(0)}{\grad f} < 0$),
then $f$ has a unique minimum in the interior of $A$. (The gradient
assumption precludes a minimum on the boundary, and the existence of
multiple minima contradicts strict convexity via the same argument as
in the Euclidean case.)

\subsection{Riemannian centre of mass}
\label{sec:riem.com}

Let $\mu$ be an unsigned measure whose support is contained within a
convex geodesic ball $B_\rho \subseteq \man$ with radius $\rho$.
Consider the energy function
$\genfctmu\colon \close{B}_{\rho} \to \R$ defined by
\begin{equation}
\label{eq:enfct}
\enfctmu{x} = \frac{1}{2} \int d_\man(x, y )^2 \,d\mu(y),
\end{equation}
where $d\mu(y)$ indcates that the integration is with respect to the
variable $y$, and the domain of integration is understood to be
$\man$, or equivalently $B_{\rho}$, since it contains the support of
$\mu$. Karcher \cite[Theorem~1.2]{Karcher} showed that $\genfctmu$ has
a unique minimum on $B_{\rho}$, provided that, if there are positive
sectional curvatures in $B_\rho$, the radius satisfies
$\rho \leq \frac{\pi}{4\sqrt{\cnstcurv}}$, where $\cnstcurv$ is an
upper bound on the sectional curvature in $B_\rho$.  This unique
minimum is called the \defn{Riemannian centre of mass} of $\mu$.

Since we will be working within these same constraints,
it is convenient to define
\begin{equation}
   \label{eq:rad.bnd}
   \rho_0 =  \min \left \{ \frac{\injradM}{2},
    \frac{\pi}{4\sqrt{\curvupbnd}} \right \},
\end{equation}
where $\curvupbnd$ is an upper bound on the sectional curvatures of
$\man$, and we understand $\frac{1}{\sqrt{\curvupbnd}}=\infty$ if
$\curvupbnd \leq 0$. Then $\rho < \rho_0$ is sufficient to enforce the
convexity condition and the curvature sensitive bound on the radius.

In the case of a discrete measure concentrated on a finite set of
points $\sigma = \{p_0,\ldots, p_k\} \subset B_{\rho}$, then instead of
$\mu$ we use $\lambda$, with $\lambda_i = \lambda(p_i)$, and 
when $\sum \lambda_i=1$
this
defines the \defn{barycentric coordinates} of the centre of mass. In
other words, the unique point $x \in B_{\rho}$ that has barycentric
coordinates $\lambda$ with respect to $\sigma$ is the point that
minimises
\begin{equation}
  \label{eq:disc.en.fct}
  \enfctbc{x} = \frac{1}{2} \sum_i \gbarycoord{i} \distM{x}{p_i}^2.
\end{equation}
The ``filled-in'' \defn{Riemannian simplex} $\riemsplxs$ defined by
$\splxs$ is defined to be the set of points in $B_\rho$ defined by all
nonnegative barycentric coordinates.

\subsection{Comparison theorems}

We denote by $\Hk = \Hk^n$ the simply connected space of dimension $n$
whose sectional curvatures are all equal to the constant
$\cnstcurv$. 
A comparison theorem provides inequalities relating geometry on a
manifold $\man$ to that on an appropriate $\Hk$. The comparison
theorems we state here are equivalent corollaries of the Rauch
comparison theorem \cite[6.4.1]{buser1981} and can be found in
\cite[6.4.3, 6.4.4]{Karcher2,buser1981}, and many other standard references.

The first comparison theorem captures the notion that geodesics
diverge more quickly in low curvature than in high curvature.
A \defn{hinge} in $\man$ consists of two geodesic segments of lengths
$\ell_1$ and $\ell_2$ meeting at a common endpoint at an angle
$\alpha$. A \defn{comparison hinge} is a hinge in $\Hk$ with the same
corresponding side lengths and angle.
Three points $w,x,y$ in a convex ball in $\man$ define a
hinge with the angle $\alpha$ at $x$ being the angle between the
unique geodesic segments connecting $w$ to $x$ and $x$ to $y$.

\begin{lem}[Hinge comparison]
  \label{lem:hinge.comparison}
  Assume the sectional curvature $\seccurv$ in $\man$ is bounded by
  $\curvlowbnd \leq K \leq \curvupbnd$, and let
  $B_{\rho} \subseteq \man$ be a convex ball
  with radius $\rho \leq \rho_0$ (as defined in \eqref{eq:rad.bnd}).
  Suppose
  $w,x,y \in B_{\rho}$ define a hinge with angle $\alpha$ at $x$, and
  let
  $\bar{w},\bar{x},\bar{y} \in \Hl$ and
  $\tilde{w},\tilde{x}, \tilde{y} \in \Hu$ define comparison hinges.

  Then
  \begin{equation*}
    d_{\Hu}(\tilde{w},\tilde{y}) \leq \distM{w}{y} \leq
    d_{\Hl}(\bar{w},\bar{y}). 
  \end{equation*}
\end{lem}

A \defn{geodesic triangle} $T$ is a collection of three points
(vertices) in $\man$, together with geodesic segments between them. A
\defn{comparison triangle} $T_{\cnstcurv}$ is a geodesic triangle
in $\Hk$ with the same edge lengths. Thus a vertex, edge, or angle in
$T$ has a corresponding vertex, edge, or angle in $T_{\cnstcurv}$.

The second comparison theorem says that triangles are fatter in spaces
with higher curvature.

\begin{lem}[Angle comparison]
  \label{lem:angle.comparison}
  Assume the sectional curvature $\seccurv$ in $\man$ is bounded by
  $\curvlowbnd \leq K \leq \curvupbnd$, and let $T$ be a geodesic
  triangle contained in a convex ball $B_{\rho} \subseteq \man$, with
  radius $\rho \leq \rho_0$ (as defined in \eqref{eq:rad.bnd}).

  Then comparison triangles $T_{\curvlowbnd}$ and $T_{\curvupbnd}$
  exist in $\Hl$ and $\Hu$. For any angle $\alpha$ in $T$, the
  corresponding angles in $T_{\curvlowbnd}$ and $T_{\curvupbnd}$
  satisfy
  \begin{equation*}
    \alpha^{}_{\curvlowbnd} \leq \alpha \leq \alpha^{}_{\curvupbnd}.
  \end{equation*}
\end{lem}

\subsection{Signed measures}
\label{sec:jordan.decomp}

In this work we are only concerned with finite (i.e., bounded) Borel
measures. We will be considering signed measures $\mu$ on $\man$.  The
Jordan decomposition theorem \cite[\S 29]{halmos1974measure}, states
that there are unique unsigned measures $\mu_+$ and $\mu_-$ such that
$\mu = \mu_+ - \mu_-$, and for our purposes we can take this as a
definition of a signed measure.

If $\nu$ is an unsigned measure on $\man$, then the \defn{support} of
$\nu$ is the closed set
\begin{equation*}
  \supp(\nu) = \{p \in \man \mid \nu(U)>0 \text{ for all open }
  U \subseteq \man \text{ with } p \in U \}.
\end{equation*}
The support of a signed measure $\mu$ is defined by $\supp(\mu) 
= \supp(\mu_+) \cup \supp(\mu_-)$.



\section{Center of mass of signed measures}
\label{sec:NegWeight}

We consider a signed measure $\mu$ with support contained in a convex
ball $B_\rho \subset \man$. In this section we 
establish constraints on $\rho$ in terms of $\mu$ and bounds on the
sectional curvatures that guarantee two properties which together are
sufficient to guarantee that $\genfctmu$ has a unique minimum in
$B_\rho$.  In \Secref{sec:Emu.convex} we show when $\genfctmu$ is
guaranteed to be strictly convex, and in \Secref{sec:grad.out} we find
conditions that ensure that the gradient of $\genfctmu$ is pointing
outwards on $\bdry{B_\rho}$. This results in the main theorem of this
section:

\begin{thm}[Centre of mass of signed measures]
  \label{thm:UniquenessKMnegative}
  Let $M$ be a manifold whose sectional curvature $K$ is bounded by
  $\curvlowbnd \leq K \leq \curvupbnd$, and let $\mu$ be a signed
  measure on $\man$, whose support is contained in a geodesic ball
  $\cballM{c}{r}$.  Suppose $B_\rho = \ballM{c}{\rho}$, with
  $r < \rho < \rho_0$ as defined in \eqref{eq:rad.bnd}.

  Then
  $\genfctmu\colon \close{B}_\rho \to \R$ (\Eqnref{eq:enfct}) has a
  unique minimum in $B_\rho$ if
  \begin{equation*}
    (\rho - r)\mup - (\rho + r)\mun>0,
  \end{equation*}
  and
  \begin{align*}
    \frac{\tu}{\tan\tu}\mup - \mun 
    &>0
    &\text{if } 0 \leq \curvlowbnd \leq \curvupbnd,\\
    \frac{\tu}{\tan\tu}\mup -
    \frac{\tl}{\tanh\tl}\mun
    &>0
    &\text{if } \curvlowbnd \leq 0 \leq \curvupbnd,\\
    \mup - \frac{\tl }{\tanh \tl}\mun 
    &>0
    &\text{if } \curvlowbnd \leq \curvupbnd \leq 0,
  \end{align*}
  where $\tl = 2\rho\sqrt{\abs{\curvlowbnd}}$, and $\tu =
  2\rho\sqrt{\abs{\curvupbnd}}$.
\end{thm}

\subsection{Ensuring the convexity of $\genfctmu$}
\label{sec:Emu.convex}

Our demonstration of the convexity of $\genfctmu$ relies on an
asymptotic expansion of the squared distance function in constant
curvature spaces. We consider points $x,y \in \Hk$, and for any point
$w$ close to $x$ we express the
distance $d_{\Hk}(w,y)$ in terms of $d_{\Hk}(x,y)$ and terms involving
powers of $\delta = d_{\Hk}(x,w)$. 
%

\begin{lem}
  \label{lem:square.dist.estimates}
  Suppose $B_\rho \subset \Hk$ is a geodesic ball of radius $\rho$,
  and $w,x,y \in B_\rho$ define a hinge with angle $\alpha$ at $x$,
  and $d_{\Hk}(x,w)= \delta$.  If $\cnstcurv>0$, assume
  $\rho < \frac{\pi}{4\sqrt{\cnstcurv}}$.

  Then 
  \begin{equation*}
    d_{\Hk}(w,y)^2 =
    \norm{\exp_x^{-1}(w) - \exp_x^{-1}(y)}^2
    + (f_\cnstcurv(\alpha,x,y) -1)\delta^2 + O(\delta^3),
  \end{equation*}
  where
  \begin{equation*}
    f_\cnstcurv(\alpha,x,y) = 
    \begin{cases}
      \cos^2\alpha + \frac{\vartheta}{\tan\vartheta}\sin^2\alpha
      &\text{if } \cnstcurv \geq 0,\\
      \cos^2\alpha + \frac{\vartheta}{\tanh\vartheta}\sin^2\alpha 
      &\text{if } \cnstcurv \leq 0,
    \end{cases}
  \end{equation*}
  and $\vartheta = d_{\Hk}(x,y)\sqrt{\abs{\cnstcurv}}$.
\end{lem}

\begin{proof}
  \def\rk{\sqrt{\cnstcurv}}
  \def\rmk{\sqrt{-\cnstcurv}}
  The result is obtained directly from a series expansion of the
  cosine rule for a space of constant curvature. Letting
  $a=d_{\Hk}(w,y)$, and $c=d_{\Hk}(x,y)=\norm{\exp_x^{-1}(y)}$, the
  cosine rules are (see e.g.,
  \cite[\S\S18.6.8, 19.3.1]{BergerGeometryII}): 
  \begin{align*}
    \cos(a\rk)
&= \cos(c\rk)\cos(\delta\rk)
+ \sin(c\rk)\sin(\delta\rk)\cos\alpha
&\text{if } \cnstcurv > 0,\\
a^2
&= c^2 + \delta^2 -2c\delta\cos\alpha
&\text{if } \cnstcurv = 0,\\
\cosh (a\rmk)
&= \cosh(c\rmk)\cosh(\delta\rmk)
- \sinh(c\rmk)\sinh(\delta\rmk)\cos\alpha
&\text{if } \cnstcurv < 0.
  \end{align*}

With the aid of a computer algebra system, we compute the series
expansion with respect to $\delta$ and find
\begin{multline*}
\left(\cnstcurv^{-1/2}\arccos\left(
\cos (c\rk)\cos(\delta\rk)
+ \sin(c\rk)\sin(\delta\rk)\cos\alpha
\right) \right)^2
\\=
c^2 -2c\delta\cos\alpha 
+ \left(\cos^2\alpha + \frac{c\rk}{\tan (c\rk)}\sin^2\alpha \right)\delta^2
+O(\delta^3),
\end{multline*}
when $\cnstcurv> 0$ and 
$c\rk$ and $\delta\rk$ are less than $\frac\pi2$, and
\begin{multline*}
\left((-\cnstcurv)^{-1/2}\arccosh\left(
\cosh(c\rmk)\cosh(\delta\rmk)
+ \sinh(c\rmk)\sinh(\delta\rmk)\cos\alpha
\right) \right)^2
\\=
c^2 -2c\delta\cos\alpha 
+ \left(\cos^2\alpha + \frac{c\rmk}{\tanh (c\rmk)}\sin^2\alpha \right)\delta^2
+O(\delta^3),
\end{multline*}
when $\cnstcurv<0$.

Since the Euclidean cosine rule in this notation yields
\begin{equation*}
\norm{\exp_x^{-1}(w) - \exp_x^{-1}(y)}^2
= c^2 - 2c\delta\cos\alpha + \delta^2, 
\end{equation*}
the result follows.
\end{proof}

Now let $\man$ be an arbitrary Riemannian manifold with sectional
curvature $K$ bounded by $\curvlowbnd \leq K \leq \curvupbnd$, and let
$B_\rho$ be a geodesic ball of radius
$\rho < \rho_0$, as defined in \eqref{eq:rad.bnd}. We consider
a hinge defined by $w,x,y \in B_\rho$, 
and let
$\bar{w},\bar{x},\bar{y} \in \Hl$ and
$\tilde{w},\tilde{x}, \tilde{y} \in \Hu$ define comparison hinges.
Then the hinge comparison theorem (\Lemref{lem:hinge.comparison}) says
\begin{equation}
  \label{eq:compare.sq.dist}
  d_{\Hu}(\tilde{w},\tilde{y})^2 \leq \distM{w}{y}^2 \leq
  d_{\Hl}(\bar{w},\bar{y})^2. 
\end{equation}

Consider now that $w = w(t)$ is a point on a unit-speed geodesic with
$w(0) = x$, so that
$\distM{x}{w(t)} = \norm{\exp_x^{-1}(w(t))} = \pm t$, and associate
the angle $\alpha$ with the positive values of $t$, i.e.,
$\bracketprod{w'(0)}{\exp_x^{-1}(y)} = \norm{\exp_x^{-1}(y)}\cos
\alpha$.
Employing \Lemref{lem:square.dist.estimates} on the bounds
\eqref{eq:compare.sq.dist}, and expanding the squared norm using the
inner product, we find
\begin{equation*}
  \distM{w(t)}{y}^2 = \norm{\exp_x^{-1}(y)}^2
  -2t\norm{\exp_x^{-1}(y)}\cos \alpha + O(t^2).
\end{equation*}
From this we conclude that
\begin{equation*}
  \frac{d}{dt}\enfctmu{w(t)}|^{}_{t=0} = - \int
  \bracketprod{w'(0)}{\exp_x^{-1}(y)} \, d\mu(y),
\end{equation*}
and since the geodesic $w$ through $x$ was arbitrary, and
\begin{equation*}
  \frac{d}{dt}\enfctmu{w(t)}|^{}_{t=0} 
= \bracketprod{w'(0)}{\grad \genfctmu(x)},
\end{equation*}
by definition of the gradient, we see that
\begin{equation}
  \label{eq:grad.enfct}
  \grad \enfctmu{x} = - \int \exp_x^{-1}(y) \,d\mu(y).
\end{equation}

We use the same technique to get a lower bound on the second derivative.
\emph{If $\mu$ were an unsigned measure}, then the lower bound in
\eqref{eq:compare.sq.dist}, would yield
\begin{equation}
  \label{eq:pos.meas.conv.bnd}
  \frac{d^2}{dt^2}\enfctmu{w(t)}|^{}_{t=0}
  \geq \int f_{\Hu}(\alpha,x,y) \,d\mu(y).
\end{equation}
Our bound
$d_{\Hu}(\tilde{x},\tilde{y}) = \distM{x}{y} < 2\rho_0 \leq
\frac{\pi}{2\sqrt{\curvupbnd}}$
ensures that $f_{\Hu}>0$ when $x$ is distinct from $y$. It follows
that if $\mu$ is an unsigned measure, then $\genfctmu$ is strictly
convex in $B_\rho$.

However, for the case of a signed measure that interests us here, more
work is required. Using the Jordan decomposition
(\Secref{sec:jordan.decomp}), our lower bound now becomes
\begin{equation}
  \label{eq:second.deriv.bnd}
  \frac{d^2}{dt^2}\enfctmu{w(t)}|^{}_{t=0}
  \geq \int f_{\Hu}(\alpha,x,y) \,d\mu_+(y)
  - \int f_{\Hl}(\alpha,x,y) \,d\mu_-(y),
\end{equation}
where we have used both of the inequalities in
\eqref{eq:compare.sq.dist}.  We wish to ensure that this bound is
strictly positive.

Observe that for $\vartheta \in [0,\frac{\pi}{2})$, the function
$\vartheta/\tan\vartheta$ is monotonically decreasing, and
$\vartheta/\tanh\vartheta$ is monotonically increasing for
$\vartheta \geq 0$. Let $z = \distM{x}{y}$. Writing
$\vartheta_{\cnstcurv}(z) = z\sqrt{\abs{\cnstcurv}}$ for
the variable that appears in the definition of $f_\cnstcurv$, we have
\begin{align*}
  \frac{\vartheta_{\cnstcurv}(2\rho)}{\tan\vartheta_\cnstcurv(2\rho)}
  &<
    \frac{\vartheta_{\cnstcurv}(z)}{\tan\vartheta_\cnstcurv(z)}
    \leq 1 &\text{if } \cnstcurv \geq 0,\\
  1 &\leq     \frac{\vartheta_{\cnstcurv}(z)}{\tanh\vartheta_\cnstcurv(z)}
      <
      \frac{\vartheta_{\cnstcurv}(2\rho)}{\tanh\vartheta_\cnstcurv(2\rho)}
      &\text{if } \cnstcurv \leq 0.
\end{align*}

Finally, by defining $\tl = 2\rho\sqrt{\abs{\curvlowbnd}}$ and
$\tu = 2\rho\sqrt{\abs{\curvupbnd}}$, and expanding $f_{\Hu}$ and
$f_{\Hl}$ in \eqref{eq:second.deriv.bnd} using these inequalities, we
obtain the convexity property of $\genfctmu$ desired for
\Thmref{thm:UniquenessKMnegative}: 

\begin{lem}
  \label{lem:enfct.convex}
  Let $\man$ be a Riemannian manifold with sectional curvature $K$
  bounded by $\curvlowbnd \leq K \leq \curvupbnd$.  Let
  $B_\rho \subseteq \man$ be a geodesic ball of radius
  $\rho < \rho_0$, and $\mu$ a signed measure with support in $B_\rho$.

  The function $\genfctmu\colon \close{B}_\rho \to
  \R$ of \Eqnref{eq:enfct} is strictly convex provided
  \begin{align*}
    \frac{\tu}{\tan\tu} \mup - \mun 
    &>0
    &\text{if } 0 \leq \curvlowbnd \leq \curvupbnd,\\
    \frac{\tu}{\tan\tu}  \mup -
    \frac{\tl}{\tanh\tl}\mun
    &>0
    &\text{if } \curvlowbnd \leq 0 \leq \curvupbnd,\\
    \mup - \frac{\tl }{\tanh \tl}\mun 
    &>0
    &\text{if } \curvlowbnd \leq \curvupbnd \leq 0,
  \end{align*}
  where $\tl = 2\rho\sqrt{\abs{\curvlowbnd}}$, and $\tu =
  2\rho\sqrt{\abs{\curvupbnd}}$.
\end{lem}

\begin{rmk}
  \Eqnref{eq:pos.meas.conv.bnd} encompasses the bound on the second
  derivative announced by Karcher \cite[(1.2.2)~and (1.2.3)]{Karcher}
  in the following sense.  If $\mu$ is an unsigned measure with
  $\mu(\man)=1$, then if $\curvupbnd<0$, \eqref{eq:pos.meas.conv.bnd}
  may be expanded as
  \begin{equation*}
  \frac{d^2}{dt^2}\enfctmu{w(t)}|^{}_{t=0}
  \geq \int \left(\cos^2\alpha
    + \frac{\vartheta}{\tanh\vartheta}\sin^2\alpha \right) \,d\mu(y)
  \geq \cos^2\alpha
    + \frac{2\rho\rl}{\tanh(2\rho\rl)}\sin^2\alpha 
    \geq 1,
  \end{equation*}
  and if $\curvupbnd>0$, then
  \begin{equation*}
  \frac{d^2}{dt^2}\enfctmu{w(t)}|^{}_{t=0}
  \geq \int \left(\cos^2\alpha
    + \frac{\vartheta}{\tan\vartheta}\sin^2\alpha  \right) \,d\mu(y)
  \geq \cos^2\alpha
    + \frac{2\rho\rl}{\tan(2\rho\rl)}\sin^2\alpha 
    \geq \frac{2\rho\rl}{\tan(2\rho\rl)}.
  \end{equation*}
  The right hand side in each case corresponds to Karcher's bound.
\end{rmk}

\subsection{The gradient of $\genfctmu$ on the boundary of a ball}
\label{sec:grad.out}

We have a signed measure $\mu$ with support in $\cballM{c}{r}$, and as
usual the sectional curvatures on $\man$ satisfy the upper bound
$\curvupbnd$. We consider a convex ball
$B_\rho = \ballM{c}{\rho} \supset \cballM{c}{r}$ and we wish to find
conditions on $\mu$ and $\rho$ such that
$\genfctmu\colon \close{B}_{\rho} \to \R$ has gradient pointing
outward on $\bdry{B_\rho}$. To be specific, let $N$ be an outward
pointing unit normal vector field on $\bdry{B_\rho}$. Then
$\grad\genfctmu$ is \defn{pointing outward} on $\bdry{B_\rho}$ if
$\bracketprod{N(x)}{\grad\enfctmu{x}}>0$ for all
$x \in \bdry{B_\rho}$.

Using the Jordan decomposition in \eqref{eq:grad.enfct}, and defining
\begin{equation*}
X(x) = -\int \exp_{x}^{-1} (y) \, d\mu_+(y), \qquad
Z(x) = -\int \exp_{x}^{-1} (y) \, d\mu_-(y).
\end{equation*}
we have $\grad \genfctmu = X - Z$, with $\bracketprod{N}{X} > 0$ and 
$\bracketprod{N}{Z} > 0$.

It is convenient to introduce $R=\rho-r>0$.  For the ``inward pointing''
part of the gradient, we have $\bracketprod{N}{Z} \leq (2r + R) \mun$,
because any given point in $\close{B}_\man(c, r)$ lies at most $2r+ R$ from a
point on $\partial B_\rho$.

The ``outward pointing'' part of the gradient satisfies the bound
$\bracketprod{N}{X} \geq R \mup$.  This bound is not as easy to prove
as the bound on the inward pointing part of the gradient.  We wish to
show that if $x\in \bdry{B_\rho}$ and $y\in \bdry{\ballM{c}{r}}$, then
\begin{equation}
  \label{eq:half.space.bnd}
  \bracketprod{-\exp_x^{-1}(y)}{N(x)} = \norm{\exp_x^{-1}(y)}\cos\varphi
  \geq R,
\end{equation}
where $\varphi$ is the angle between $-N(x)$ and $\exp_x^{-1}(y)$. In
other words, we want to show that the image of $\cballM{c}{r}$ under
$\exp_x^{-1}$ is contained in the half-space
$\{v \in T_x\man \mid \bracketprod{v}{-N(x)}\geq R \}$.

Let $b = \norm{\exp_x^{-1}(y)}$ be the length of the geodesic segment
between $x$ and $y$. We show that, if $b \leq R/ \cos \varphi$, the
only geodesic triangle with edge lengths $r$, $r+R$, $b$, vertices
$c$, $x$, $y$, and angle $\varphi$ at $x$, as in
Figure~\ref{Figure:KMAngle}, is the trivial triangle with
$\varphi =0$.

Observe that $ \sin (a s) \leq a \sin s $ for $a \geq 1$ and
$0 \leq as \leq \pi/2$, and this inequality is strict if $s>0$ and
$a>1$. This can be seen by differentiating with respect to $s$ and
observing that $\cos a s \leq \cos s$ under the same conditions. 

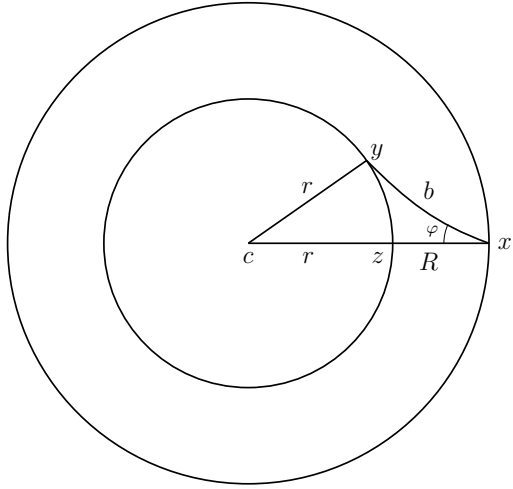
\begin{figure}
\centerline{
\scalebox{.8}{
\large
\begin{tikzpicture}
\coordinate [label={above :$b$}] (b) at (3,0.6);
\coordinate [label={below :$c$}] (p) at (0,0);
\coordinate [label={right :$x$}] (x) at (4,0);
\coordinate [label={below left :$z$}] (z) at (2.4,0);
\draw[thick,black] (0,0) circle (4cm);
\draw[thick,black] (0,0) circle (2.4cm);
\draw[thick] (35:2.4) to [out=-45,in=160] (x) ;
\draw (4,0) ++ (180:0.75) arc (180 :155 : 0.75) ;
\coordinate [label={left :{\scriptsize $\varphi$}}] (t) at (3.3,0.2);
\draw[thick] (0,0) -- (x) node[near start, below] {$r$} node[near end,
below] {$R$};
\draw[thick] (0,0) -- (35:2.4) node [midway, above] {$r$} ;
\node[above] at (30:2.47) {$y$};
\end{tikzpicture}
}}
\caption[Two geodesic balls]{Sketch for the argument that
  $b\cos\varphi \geq R=\rho-r$.  This is a schematic view in Riemann normal
  coordinates whose origin coincides with the centre of the concentric
  geodesic balls $\ballM{c}{r}$ and $\ballM{c}{\rho}$.}
\label{Figure:KMAngle}
\end{figure}

Returning to the geodesic triangle, the angle comparison theorem
(\Lemref{lem:angle.comparison}) implies that $\cos \varphi$ is bounded
from below by $\cos \varphi^{}_\curvupbnd$, where
$\varphi^{}_\curvupbnd$ denotes the angle corresponding to $\varphi$
for the comparison triangle in the space of curvature
$\curvupbnd$. Without loss of generality we can assume that
$\curvupbnd$ is positive (for our purposes here, nothing is gained by
requiring $\curvupbnd$ to be a \emph{least} upper bound on the
sectional curvatures).

We have $b\geq R$, by the triangle inequality. Using our comparison
triangle, and letting $\tr = r\sqrt{\curvupbnd}$,
$\tR = R\sqrt{\curvupbnd}$, and $\tb = b\sqrt{\curvupbnd}$,
the cosine rule for spaces of positive curvature now yields
\begin{align*}
\cos \tr &= \cos (\tr+\tR) \cos (\tb) + \sin (\tr+\tR ) 
\sin (\tb) \cos \varphi^{}_\curvupbnd
\\
&\leq \cos (\tr+\tR) \cos (\tb) + \sin (\tr+\tR ) 
\sin (\tR/ \cos \varphi^{}_\curvupbnd)
\cos  \varphi^{}_\curvupbnd
\\
&\leq \cos (\tr+\tR) \cos (\tb) + \sin (\tr+\tR ) \sin (\tR)
\\
&\leq \cos (\tr+\tR) \cos (\tR) + \sin (\tr+\tR ) \sin (\tR)
\\
&= \cos \tr ,
\end{align*}
and since the inequalities are strict unless
$ \varphi^{}_\curvupbnd=0$, we must have $\varphi=0$. This confirms
\Eqnref{eq:half.space.bnd}, and we obtain our desired result:

\begin{lem}
  Let $\mu$ be a signed measure on $\man$ with support in
  $\cballM{c}{r}$, which is contained in a concentric convex geodesic
  ball $B_\rho = \ballM{c}{\rho}$. The gradient of $\genfctmu$ is
  pointing outward on $\bdry{B_\rho}$ if
\begin{equation*}
(\rho-r)\mup - (\rho + r)\mun>0.
\end{equation*}
\end{lem}


%
%

\section{Barycentric coordinate neighbourhoods}
\label{sec:barycoords}

As mentioned in \Secref{sec:riem.com}, the centre of mass construction
allows us to define barycentric coordinates on Riemannian
manifolds. The signed measure $\mu$ is replaced with a discrete signed
measure $\lambda$ that assigns ``weights'' to the points in a 
set $\splxs$ of $n+1$ points contained in a sufficiently small ball
$B_\rho \subseteq \man$. The associated Riemannian simplex $\riemsplxs$
is the set of points with nonnegative barycentric coordinates.
\Thmref{thm:UniquenessKMnegative} allows us to define a barycentric
coordinates that extend to an open neighbourhood containing
$\riemsplxs$. 

In this section we quantify the extent of these barycentric coordinate
neighbourhoods. To facilitate this we first make some mild simplifying
assumptions to the setting of \Thmref{thm:UniquenessKMnegative},
arriving at \Corref{cor:simple.signed.measure}, which provides
transparent upper and lower bounds on $\rho$ in terms of the total
negative mass, and allows for easy comparison with Karcher's result
\cite[Theorem~1.2]{Karcher}. Then in \Secref{sec:nbrhd.extent} we use
these simplified bounds to quantify the extent of well-defined
barycentric coordinate neighbourhoods. The size of these
neighbourhoods depends on a measure of the \emph{quality} of $\splxs$,
as discussed in \Secref{sec:nbrhd.extent}. 

Our main result, \Thmref{thm:bary.coord.nbhd.extent}, quantifies the
size of a ball in the parameter domain, $B\subset \R^n$, on which the
barycentric coordinate map $b\colon B \to B_\rho \subseteq \man$ is
well defined. By earlier work \cite{RiemSimp2015} we are ensured that
the barycentric coordinate map is not only well defined, but is in
fact an embedding on this domain. We also estimate the size of a
geodesic ball that is contained in the barycentric coordinate
neighbourhood $b(B)$ (\Propref{prop:contained.ball}). In
\Secref{sec:shared.facet} we exploit barycentric coordinate
neighbourhoods to show that Riemannian simplices behave like their Euclidean
counterparts: simplices with the same orientation that share a facet
intersect only in that facet.

\subsection{Simplifying the setting}
\label{sec:simplifying}

Our goal now is to simplify the constraints imposed in
\Thmref{thm:UniquenessKMnegative}, albeit at the expense of weakening
the statement. We will assume $\mu(\man)=1$; this normalisation was
not required for \Thmref{thm:UniquenessKMnegative}, although we did
assume $\mu_-(\man),\mu_+(\man) < \infty$. We will also replace
$\curvlowbnd$ and $\curvupbnd$ with a single bound $\curvabsbnd$ on
the absolute value of the sectional curvatures. Before exploiting
these assumptions, we first replace the bounds of
\Lemref{lem:enfct.convex} with constraints linear in $\rho$.

\begin{lem}
  \label{lem:lin.bnd.tan}
  We have
  \begin{align*}
    &&1 - \frac{2\vartheta}{\pi} &\leq \frac{\vartheta}{\tan\vartheta}
&&\text{for } \vartheta \in [0,\tfrac{\pi}{2}),&&\\
    &&1 + \vartheta &\geq \frac{\vartheta}{\tanh\vartheta}
&&\text{for } \vartheta \geq 0.&&\\
  \end{align*}
\end{lem}
\begin{proof}
  By taking the second derivatives, and observing that
$\frac{\vartheta}{\tan\vartheta} \geq 1$ and
$\frac{\vartheta}{\tanh\vartheta} \leq 1$ in the stated ranges, we
find that $\frac{\vartheta}{\tan\vartheta}$ is concave and
$\frac{\vartheta}{\tanh\vartheta}$ is convex in their ranges. The
inequalities now follow from Jensen's inequality, i.e., by comparing
the function to a secant line. 
\end{proof}

Revisiting \Lemref{lem:enfct.convex} with the assumption that
$\curvlowbnd = -\curvabsbnd$ and $\curvupbnd = \curvabsbnd>0$, the
convexity of $\genfctmu$ is ensured if
 the middle inequality is satisfied.
Using \Lemref{lem:lin.bnd.tan}, and then expanding $\tl=\tu = 2\rho\rl$ we
get that $\genfctmu$ is convex if
\begin{equation*}
  \left(1-\frac{4\rho\rl}{\pi} \right)  \mup -
  (1 + 2\rho\rl)\mun
  >0.
\end{equation*}
Now using the assumption that $\mu(\man)= \mup - \mun = 1$, we obtain
a convenient simplified version of \Thmref{thm:UniquenessKMnegative}.
\begin{cor}[Centre of mass of signed measures]
  \label{cor:simple.signed.measure}
  Suppose the sectional curvature $K$ of $\man$ is bounded by
  $\abs{K} \leq \curvabsbnd$, and $\mu$ is a signed measure on $\man$
  with $\mu(\man)=1$ and support contained in a geodesic ball
  $\cballM{c}{r}$. If $B_\rho = \ballM{c}{\rho}$ with
  $r< \rho < \rho_0$ as defined in \eqref{eq:rad.bnd}, then
  $\genfctmu\colon \close{B}_\rho \to \R$ (\Eqnref{eq:enfct}) has a
  unique minimum in $B_\rho$ if
  \begin{align}
    \label{eq:rR.constraint}
    \rho &> (1 + 2\mun)r,\\
\intertext{and}
    \label{eq:curv.constraint}
    \rho &< \frac{\pi}{4\rl}\left(1 + C\mun \right)^{-1},
  \end{align}
  where $C= \left(1+\frac{\pi}{2} \right)$.
\end{cor}

If $\mu$ is an unsigned measure, the conditions of
\Corref{cor:simple.signed.measure} reduce to the same conditions
required in Karcher's theorem \cite[Thm.~1.2]{Karcher}. 

\subsection{The extent of barycentric coordinate neighbourhoods}
\label{sec:nbrhd.extent}

We consider a set $\splxs$ of $n+1$ points contained in a geodesic
ball $B_\rho \subseteq \man$ of radius $\rho < \rho_0$ 
($\rho$ will be further constrained below).  We will require that there exists
a Euclidean $n$-simplex $\tsplxs \subset \E^n$ whose vertices are in
correspondence with the points of $\splxs$ and whose edge lengths are
equal to the geodesic distances between the corresponding points in
$\splxs$.  This in itself is a constraint on $\splxs$ (see
\cite[\S5.1]{RiemSimp2015}).

Let $(\lambda_i)$ be the barycentric coordinate functions associated
with $\tsplxs$. We define $\lambda_+ \colon \E^n \to \R$ to be the
function that gives the sum of the positive barycentric coordinates,
and $\lambda_-$ to be the sum of the negative coordinates, so
$\lambda_+ + \lambda_- = 1$ is constant.  Given any fixed
$u \in \E^n$, the barycentric coordinates $(\lambda_i(u))$ define a
discrete measure on $\man$ (a weight for each point of $\splxs$) via
the correspondence between $\splxs$ and the vertices of $\tsplxs$.  In
this context, $\abs{\lambda_-(u)}$ plays the role of $\mu_-(\man)$ in
\Corref{cor:simple.signed.measure}.

We would like to find a domain $B \subset \E^n$, with $\tsplxs \subset
B$, such that the constraints of \Corref{cor:simple.signed.measure} are
satisfied for any $u\in B$. So we are interested in finding a bound on
$\abs{\lambda_-}$.

We choose an arbitrary vertex $\tv \in \tsplxs$, and we will find a
radius $\tr$ such that for all points $u$ in $B=\ballEn{\tv}{\tr}$,
this goal is attained. 

Let $\gsplxalt$ be a lower bound on the altitudes of $\tsplxs$.  The
magnitude of the gradient of each barycentric coordinate function
$\lambda_i$ is bounded by $a^{-1}$. Therefore $-\lambda_i < \tr/a$ on
$B$.  Since at most $n$ coordinate functions can be negative at any
point in $B$, we find
\begin{equation}
  \label{eq:tr.bnd.lambda.neg}
  \abs{\lambda_-} < \frac{n\tr}{\gsplxalt} = \frac{\tr}{\glongedge
    \gthickness}, 
\end{equation}
where $\glongedge$ is a strict upper bound on the edge lengths, and
$t = \gsplxalt/n\glongedge$ is the \defn{thickness} of $\tsplxs$.

We choose $r=\glongedge$ as the radius of
a ball containing $\splxs \subset \man$.  Observing that
$\abs{\lambda_-} = \lambda_+ - 1$ and exploiting
\eqref{eq:tr.bnd.lambda.neg} we find that
\eqref{eq:rR.constraint} is satisfied if
\begin{equation}
  \label{eq:rho.low}
  \rho > \glongedge + \frac{2\tr}{t}.
\end{equation}
Similarly, \eqref{eq:curv.constraint} may be replaced with
\begin{equation}
  \label{eq:rho.high}
  \rho < \frac{\pi}{4\sqrt{\curvabsbnd}}
  \left(1 + C\frac{\tr}{Lt} \right)^{-1}, 
\end{equation}
where $C = 1 + \pi/2$.

In order for there to exist a $\rho$ satisfying both of these
constraints, we require
\begin{equation*}
\Big(\glongedge + \frac{2\tr}{t}\Big)
\Big(1 + C\frac{\tr}{Lt}\Big)
<
\frac{\pi}{4\sqrt{\curvabsbnd}},
\end{equation*}
or
\begin{equation*}
  \frac{2C}{Lt^2}\tr^2 + \frac{2+C}{t}\tr + \Big(L -
  \frac{\pi}{4\sqrt{\curvabsbnd}} \Big) < 0. 
\end{equation*}
We require $\tr$ to be strictly smaller than the largest root:
\begin{equation}
  \label{eq:raw.tr.bnd}
  \tr < \frac{Lt(2+C)}{4C}
  \bigg({-}1 + \bigg(1 +
  \frac{8C}{(2+C)L}\Big(\frac{\pi}{4\sqrt{\curvabsbnd}} - L \Big) 
  \bigg)^{1/2} \bigg).
\end{equation}
Since $\frac{4}{9} < \frac{2+C}{4C} < \frac{1}{2}$,
\Eqnref{eq:raw.tr.bnd} is satisfied when
\begin{equation}
  \label{eq:clean.tr.bnd}
  \tr \leq \frac{Lt}{3}\bigg({-}1 + \sqrt{-3
    +\frac{\pi}{4L\sqrt{\curvabsbnd}}} \bigg). 
\end{equation}
Observe that this bound is positive provided
$L \leq \frac{1}{6\sqrt{\curvabsbnd}} <
\frac{\pi}{16\sqrt{\curvabsbnd}}$.
Also, the bound \eqref{eq:clean.tr.bnd} is maximised when
$L= \frac{\pi}{48\sqrt{\curvabsbnd}} \sim
\frac{1}{15\sqrt{\curvabsbnd}}$.

We want to ensure that the entire simplex $\tsplxs$ is contained in
the domain of the barycentric coordinates. Therefore, using a scale
parameter $s \geq 1$, we set $\tr = sL$ in \eqref{eq:clean.tr.bnd} and
find that $L$ must satisfy
\begin{equation*}
  L \leq \frac{\pi}{4\sqrt{\curvabsbnd}}\Big( \frac{9s^2}{t^2} +
  \frac{6s}{t} + 4 \Big)^{-1}. 
\end{equation*}
Since $t<1$, and $s\geq 1$ this is satisfied when
\begin{equation}
  \label{eq:splx.size.bnd}
  L \leq \frac{t^2}{25s^2\sqrt{\curvabsbnd}}.
\end{equation}
This is the criterion bounding $\tr$ that we sought. 

Thus when \eqref{eq:splx.size.bnd} is satisfied, the barycentric
coordinate map $b\colon \ballEn{\tv}{sL} \to B_{\rho}$ is well
defined, provided $\rho$ satisfies \eqref{eq:rho.low} and
\eqref{eq:rho.high}, with $\tr=sL$; the existence of such a $\rho$ is
assured.  But we can say more. It has been shown
\cite[Prop.~29]{RiemSimp2015} that the differential of the barycentric
coordinate map is nondegenerate provided
$L \leq t/(3\sqrt{\curvabsbnd})$, which is obviously satisfied when
\eqref{eq:splx.size.bnd} is. Although the argument was made in the
context of positive weights, this fact was not used. What is required
is that the energy functional \eqref{eq:disc.en.fct} that defines the
barycentric coordinate map be convex, but we still have that in our
current context. Thus we obtain our main result:

\begin{thm}[Barycentric coordinate neighbourhoods]
  \label{thm:bary.coord.nbhd.extent}
  Let $\man$ be an $n$-dimensional Riemannian manifold with sectional
  curvature $K$ bounded by $\abs{K} \leq \curvabsbnd$, and let $\sigma
  \subset \man$ be a set of $n+1$ points such that $\distM{p}{q}<L$
  for any $p,q \in \sigma$.

  Given a scale parameter $s\geq 1$, if $\sigma$ defines a Euclidean
  simplex $\tsplxs \subset \E^n$ with the same edge lengths, and with
  thickness $t$ satisfying
  \begin{equation*}
    t^2 \geq 25s^2L\rl,
  \end{equation*}
  then, 
  the barycentric coordinate map
  $b\colon \ballEn{\tv}{sL} \to B_{\rho}$, where $\tv$ is any vertex of
  $\tsplxs$, is well defined, and is in fact an embedding. The
  ball $B_\rho$ is centred on the vertex of $\splxs$ corresponding to
  $\tv$, and its radius $\rho$ must lie in the nonempty interval
  given by
  \begin{equation*}
    \left(1 + \frac{2s}{t} \right)L < \rho
    < \frac{\pi}{4\rl}\left(1 + C\frac{s}{t} \right)^{-1}\!,
  \end{equation*}
  where $C= 1 + \pi/2$.
\end{thm}

Although the image of $\ballEn{\tv}{sL}$ under the barycentric
coordinate map $b$ is guaranteed to be contained in $B_\rho$ for any
sufficiently large $\rho$, it may be also useful to know the size of a
geodesic ball that is contained in $b(\ballEn{\tv}{sL})$. The metric distortion
of the barycentric coordinate map has already been calculated
\cite[\S5.2]{RiemSimp2015}, and although this argument was made in the
context of non-negative weights, this fact was not exploited, and the
argument carries verbatim to the case of interest here. The argument
does require that $\rho$ be bounded by $\rho \leq \frac{t}{6\rl}$ but
this is satisfied if the conditions of
\Thmref{thm:bary.coord.nbhd.extent} are satisfied.

\begin{lem}[{{\cite[\S5.2]{RiemSimp2015}}}]
  \label{lem:bary.distort}
  Given the assumptions of \Thmref{thm:bary.coord.nbhd.extent}, the
  barycentric coordinate map $b\colon \ballEn{\tv}{sL} \to B_\rho$
  satisfies
  \begin{equation*}
    \abs{\distM{b(u)}{b(w)} - \norm{u-w}}
    \leq \frac{50\curvabsbnd\rho^2}{t^2} \norm{u-w},
  \end{equation*}
  for any $u,w \in \ballEn{\tv}{sL}$.
\end{lem}

There is no reason to choose a $\rho$ larger than necessary; with the
assumptions of \Thmref{thm:bary.coord.nbhd.extent}, we can choose
$\rho^2 = \left(1 + \frac{2s}{t} \right)^2L^2 + \varepsilon^2 \leq
\frac{3^2t^2}{25^2s^2\curvabsbnd}$,
where $\varepsilon$ is an arbitrarily small distance. Employing this
bound in \Lemref{lem:bary.distort}, and choosing a point
$w \in \bdry{\ballEn{\tv}{sL}}$, we find
\begin{equation*}
  \distM{b(\tv)}{b(w)} \geq \left(1 - \frac{18}{25s^2} \right)sL. 
\end{equation*}

We desire that the image of $\tsplxs$ be contained in this
ball. Demanding $\left(s - \frac{18}{25s} \right)L \geq L$, we find
this will be satisfied provided $s \geq \frac{3}{2}$, and we have the
following addendum to \Thmref{thm:bary.coord.nbhd.extent}:

\begin{prop}
  \label{prop:contained.ball}
  Assume that the conditions of \Thmref{thm:bary.coord.nbhd.extent} are
  satisfied, and that $v \in \splxs$ is the vertex corresponding to
  $\tv \in \tsplxs$. Then the image of $\ballEn{\tv}{sL}$ under the
  barycentric coordinate map contains the geodesic ball
  $\ballM{v}{\bar{r}}$, where
  \begin{equation*}
    \bar{r} = \left(s - \frac{18}{25s} \right)L.
  \end{equation*}
  If $s \geq \frac{3}{2}$, then $\bar{r} \geq L$, and so
  $\ballM{v}{\bar{r}}$ contains $\tsplxs$.
\end{prop}

\subsection{Riemannian simplices that share a facet}
\label{sec:shared.facet}

Suppose $\riemsplxs \subset \man$ is a Riemannian $n$-simplex defined
by $\splxs = \{p,p_0,\ldots, p_{n-1}\}$. The \defn{facet} of
$\riemsplxs$ opposite $p$ is the Riemannian $(n-1)$-simplex
$\riemsplxeta$ defined by $\eta=\{p_0,\ldots,p_{n-1}\}$, i.e., the
set of points in $\riemsplxs$ where the barycentric coordinate
associated with $p$ is $0$. Suppose now that $\riemsplxt \subset \man$
is another Riemannian simplex defined by $\splxt
=\{p_0,q,p_1,\ldots,p_{n-1}\}$, so that $\riemsplxt$ shares the facet
$\riemsplxeta$ with $\riemsplxs$. 

Suppose that both $\riemsplxs$ and $\riemsplxt$ satisfy the conditions
of \Thmref{thm:bary.coord.nbhd.extent}, with scale factor $s=3/2$, and
with $L$ being a common strict upper bound on their edge lengths. Then
they are nondegenerate, and we can define their orientation in
$B_\rho$ (see {{\cite[\S3.3]{RiemSimp2015}}}). If they have the same
orientation, then they can be represented by Euclidean simplices
$\tsplxs =\{\tp,\tp_0,\ldots, \tp_{n-1}\}$ and
$\tsplxt = \{\tp_0,\tq, \tp_1, \ldots,\tp_{n-1}\}$ that share a common
facet, $\tilde{\eta} = \tsplxs \cap \tsplxt$, such that $\tp$ and
$\tilde{q}$ lie on opposite sides of the hyperplane
$\affhull{\tilde{\eta}}$. Thus $\tsplxtE = \convhull{\tsplxt}$ lies in
the region where the barycentric coordinate associated with
$p\in \tsplxs$ is nonpositive, and similarly,
$\tsplxsE = \convhull{\tsplxs}$ lies in the region where the
barycentric coordinate associated with $\tq\in \tsplxt$ is
nonpositive, and $\affhull{\tilde{\eta}}$ is the common 0 level set of
both of these barycentric coordniate functions.

The barycentric coordinate map for $\splxs$ is an embedding of
$\ballEn{\tp_0}{sL}$ that maps $\tsplxsE$ to $\riemsplxs$, preserving
the associated barycentric coordinates, and similarly the barycentric
coordinate map for $\splxt$ maps $\tsplxtE$ to $\riemsplxt$. These
maps agree on $\affhull{\tilde{\eta}}$, and by our choice of scale
factor $s$, \Propref{prop:contained.ball} ensures that
$\ballM{p_0}{L}$ is contained in the image of both maps.  Since the
$(n-1)$-submanifold that is the image of $\affhull{\tilde{\eta}}$
separates this ball into two components according to the sign of the
barycentric coordinate of $p$ (and likewise of $q$), it follows that
$\riemsplxs \cap \riemsplxt = \riemsplxeta$.



\paragraph{Acknowledgements}
\Thmref{thm:UniquenessKMnegative} appeared in Wintraecken's PhD thesis
\cite[Theorem~3.4.9]{wintraecken2015}. The current exposition
benefited from the comments and suggestions of the jury. In particular
we thank John Sullivan for valuable feedback on this work.
This research has been partially supported by the 7th Framework
Programme for Research of the European Commission, under FET-Open
grant number 255827 (CGL Computational Geometry Learning).
Partial support has been provided by the Advanced Grant of the
European Research Council GUDHI (Geometric Understanding in Higher
Dimensions).


\phantomsection
\bibliographystyle{alpha}
\addcontentsline{toc}{section}{Bibliography}
\bibliography{geomrefs}

\begin{thebibliography}{DVW15}

\bibitem[Ber87]{BergerGeometryII}
M.~Berger.
\newblock {\em Geometry {II}}.
\newblock Universitext. Springer-Verlag, 1987.

\bibitem[BK81]{buser1981}
P.~Buser and H.~Karcher.
\newblock {\em {G}romov's almost flat manifolds}.
\newblock Number~81 in Ast\'erisque. Soci{\'e}t{\'e} math{\'e}matique de
  France, 1981.

\bibitem[Cha06]{Chavel}
I.~Chavel.
\newblock {\em {R}iemannian Geometry: A Modern Introduction}.
\newblock Number~98 in Cambridge studies in advanced mathematics. Cambridge
  University Press, 2006.

\bibitem[DVW15]{RiemSimp2015}
R.~Dyer, G.~Vegter, and M.~Wintraecken.
\newblock Riemannian simplices and triangulations.
\newblock {\em Geometriae Dedicata}, 179(1):91--138, 2015.

\bibitem[Gro04]{groisser2004}
D.~Groisser.
\newblock {N}ewton's method, zeroes of vector fields, and the {R}iemannian
  center of mass.
\newblock {\em Adv. in Appl. Math.}, 33(1):95--135, 2004.

\bibitem[Hal74]{halmos1974measure}
Paul~R Halmos.
\newblock {\em Measure theory}.
\newblock Number~18 in GTM. Springer, 1974.

\bibitem[Kar77]{Karcher}
H.~Karcher.
\newblock Riemannian center of mass and mollifier smoothing.
\newblock {\em Communications on Pure and Applied Mathematics}, 30:509--541,
  1977.

\bibitem[Kar89]{Karcher2}
H.~Karcher.
\newblock Riemannian comparison constructions.
\newblock In S.S. Chern, editor, {\em Global Differential Geometry}, pages
  170--222. The mathematical association of America, 1989.

\bibitem[Kar14]{karcher2014hist}
H.~Karcher.
\newblock {R}iemannian {C}entre of {M}ass and so called karcher mean.
\newblock Historical Note 1407.2087, arXiv, 2014.

\bibitem[Ken90]{Kendall}
W.S. Kendall.
\newblock Probability, convexity, and harmonic maps with small image {I}:
  Uniqueness and fine existence.
\newblock {\em Procedings of the London Mathematical society}, s3-61 (Issue
  2):371--406, 1990.

\bibitem[San16]{Sander2}
O.~Sander.
\newblock Geodesic finite elements of higher order.
\newblock {\em IMA Journal of Numerical Analysis}, 36(1):238--266, 2016.

\bibitem[Win15]{wintraecken2015}
M.H.M.J. Wintraecken.
\newblock {\em Ambient and intrinsic triangulations and topological methods in
  cosmology}.
\newblock PhD thesis, Rijksuniversiteit Groningen, 2015.
\newblock
  \url{https://www.rug.nl/research/portal/files/23072604/Complete_thesis.pdf}.

\end{thebibliography}

\end{document}